\documentclass{article}
\usepackage[T1]{fontenc}
\usepackage{a4wide}
\usepackage{authblk}
\usepackage{graphicx} 
\usepackage{systeme}
\usepackage{amsmath}
\usepackage{amssymb}
\usepackage[hidelinks]{hyperref}
\newcommand{\citep}{\cite}
\usepackage{amsthm}
\usepackage{float}
\usepackage{tikz}
\usepackage{xcolor}

\usepackage{enumitem}

\newtheorem{thm}{Theorem}
\newtheorem{assum}[thm]{Assumption}

\newtheorem{lem}[thm]{Lemma}

\DeclareMathAlphabet{\mathbbb}{U}{bbold}{m}{n}
\renewcommand{\leq}{\leqslant}

\renewcommand{\geq}{\geqslant}

\newcommand{\Ubar}{\Bar{U}}

\newcommand{\mycomment}[1]{}

\newcommand{\R}{\mathbb{R}}
\newcommand{\N}{\mathbb{N}}
\newcommand{\Z}{\mathbb{Z}}
\newcommand{\C}{\mathbb{C}}

\title{

Stabilization of a chain of 3 hyperbolic PDEs with 2 inputs in arbitrary position\thanks{This project received funding from the Agence Nationale de la Recherche via grant PANOPLY ANR-23-CE48-0001-01.}}
\author[1]{Adam Braun}
\author[1]{Jean Auriol}
\author[1]{Lucas Brivadis}
\affil[1]{Université Paris-Saclay, CNRS, CentraleSupélec, Laboratoire des Signaux et Systèmes, 91190, Gif-sur-Yvette, France. Emails:
        {\tt\small adam.braun@centralesupelec.fr; jean.auriol@centralesupelec.fr; lucas.brivadis@centralesupelec.fr}}%
\date{}
\begin{document}

\maketitle
\begingroup
\renewcommand\thefootnote{}\footnotetext{%
\footnotesize This paper has been accepted to \emph{IFAC WC 2026 (Busan)}.}
\endgroup
\begin{abstract}This paper addresses the stabilization of a chain of three coupled hyperbolic partial differential equations actuated by two control inputs applied at arbitrary nodes of the network. With the exception of configurations where one input is located at an endpoint, cases already well studied in the literature, all admissible two-input configurations are treated in this paper within a unified framework.
The proposed approach relies on a backstepping transformation combined with a reformulation of the closed-loop dynamics as an Integral Difference Equation (IDE). This IDE representation reveals a common structural pattern across configurations and clarifies the role played by delayed dynamics in the stability analysis. Within this formulation, the stabilization problem can be handled using existing IDE control techniques.
For most configurations, the stabilization of the PDE system requires an approximate spectral controllability assumption. Remarkably, one specific configuration can be stabilized without imposing any additional spectral condition. In contrast, we also provide an explicit example of a configuration for which the required spectral controllability property fails to hold.
\end{abstract}


\section{Introduction}
The controllability and stabilization of networks of hyperbolic Partial Differential Equations (PDEs) constitute a fundamental area of research with wide-ranging applications in engineering. Hyperbolic PDEs model systems governed by wave propagation and transport phenomena, such as fluid dynamics~\citep{HayatShang2021, Jean_Adam_Ouidir_Mathieu}, electrical transmission networks~\citep{Arrillaga1998}, and mechanical or structural vibrations~\citep{aarsnes2018torsional}.
In many practical settings, these systems are interconnected through network structures in which the dynamics of each subsystem influence, and are influenced by, neighboring components through boundary interactions. Such configurations arise, for instance, in traffic-flow management~\citep{YuKrstic2023}, where the propagation of density and velocity waves along interconnected road segments plays a central role.

Considerable attention has been devoted to the control of cascaded hyperbolic PDE systems, including ODE--PDE--ODE configurations. Most constructive designs rely on the backstepping method~\citep{deutscher2021backstepping,Deutscher2018,wang2020delay,irscheid2023output,auriol2023robustification}, with recent extensions based on infinite-dimensional dynamic augmentation~\citep{GEHRING2025112032}. 
Backstepping has proven effective for stabilizing chains of interconnected hyperbolic PDEs, initially for scalar subsystems~\citep{deutscher2021backstepping} and more recently for non-scalar ones~\citep{auriol2024output}. A common feature of recent contributions is the reformulation of the closed-loop dynamics as a \emph{time-delay system}, which serves as the basis for controller design~\citep{AuriolHDR}.

Most existing results assume that the control input is applied at one end of the chain. This configuration covers numerous applications, such as drilling systems or UAV--cable--payload structures~\citep{wang2020delay}, but does not reflect scenarios where the actuator lies at an interior node of the network. In traffic-flow control, for instance, ramp-metering devices may be located at junctions. 
For the case of two hyperbolic PDEs coupled at a junction, a stabilizing controller based on Fredholm transformations was proposed in~\citep{Redaud2022}. However, this approach does not scale well to chains with more than two PDEs.

To address these limitations, a recent methodology was introduced in~\citep[Chapter~9]{AuriolHDR} and~\citep{AuriolLCSS}. 
Using backstepping transformations, the PDE interconnection is reformulated as an Integral Difference Equation (IDE), which allows one to design an explicit stabilizing control law expressed in terms of integrals of the state and input histories. 
This approach was extended in~\citep{braun2025stabilizationchainhyperbolicpdes} to a chain of three hyperbolic PDEs actuated by a single input. 
In that setting, stabilization was achieved under a spectral controllability assumption, namely the absence of common zeros between two holomorphic functions arising in the IDE representation. From a practical and cost-oriented perspective, it is important to study the stabilization of the chain under different input configurations. 
Indeed, the placement and number of actuators have a direct impact on implementation cost, and understanding which configurations allow for stabilization provides valuable guidance for efficient actuator design and allocation.

In this paper, we investigate the stabilization of the same three-PDE chain for all configurations involving two boundary inputs. 
We omit the configurations in which one input is applied at an endpoint of the chain, since this configuration has already been treated in~\citep{auriol2024output}. 
For the remaining configurations, and even with two control inputs, the proposed stabilization methodology requires, in all but one case, a spectral controllability assumption stating that three associated holomorphic functions have no common zeros.
Although it is weaker than the spectral controllability condition required in the single-input case~\citep{braun2025stabilizationchainhyperbolicpdes}, it does not hold universally; in particular, we exhibit a specific configuration for which the assumption fails.

One configuration can be straightforwardly stabilized using the methodology of~\citep{auriol2024output}. For the others, using the backstepping transformation from~\citep{braun2025stabilizationchainhyperbolicpdes}, the original system is mapped to an intermediate target system of transport equations with modified boundary conditions.
The intermediate target system is then rewritten as an IDE for which an explicit stabilizing control law can be constructed following the methodology of~\citep{AuriolLCSS}. 

The structure of the paper follows this strategy. 
Section~\ref{Problem formulation} introduces the problem setting, the set of two-input configurations, and the standing assumptions. 
Section~\ref{Design:feedback} details the reformulation of the PDE system into an IDE for the remaining configurations.
Finally, Section~\ref{contre_exemple} provides a configuration for which the required spectral assumption is not satisfied.


\section{Problem Formulation}
\label{Problem formulation}
Let $L^2(a, b)$ for $a<b$ be the space of real valued square integrable functions over $(a, b)$,
and let $H^1(0, 1)=\{f\in L^2(0, 1): f'\in L^2(0, 1)\}$.
In this paper, we consider a chain of three subsystems interconnected through their boundaries. All configurations involving two inputs acting on the boundary of the chain are examined.  However, we deliberately omit the configurations in which one input is applied at an endpoint of the chain, since this problem has already been addressed in~\citep{auriol2024output}. Although the system is defined with several inputs $\Ubar_j$  the subsequent stability analysis is carried out under the assumption that exactly two inputs are active, with all others set to zero. 
Every subsystem $i$ is represented by a set of PDEs. For  all $i\in \{1,2,3\},$ the states $u_i, v_i$ satisfy
\begin{equation}
\label{system_edp_origin}
\begin{aligned} 
    &\partial_t u_i(t,x) + \lambda_i \partial_x u_i(t,x) = \sigma^{+}_i(x) v_i(t,x),\\
     &\partial_t v_i(t,x) - \mu_i \partial_x v_i(t,x) = \sigma^{-}_i(x) u_i(t,x),\\
&u_i(t,0) =  \delta_{i,2}\Ubar_4(t)+ \delta_{i,3}\Ubar_3(t)+q_{ii}v_{i}(t,0) +q_{i i-1}u_{i-1}(t,1)\\
&v_i(t,1) =\delta_{i,1}\Ubar_1(t) +\delta_{i,2}\Ubar_2(t)+\rho_{ii}u_i(t,1) + \rho_{ii+1}v_{i+1}(t,0)
\end{aligned}
\end{equation}
where $\delta_{i,j}$ is the Kronecker delta, and with $q_{10} = \rho_{34} = 0$, $t\geq 0$ and $x\in [0, 1]$. The velocities $\lambda_i$ and $\mu_i$ are positive real constants. The in-domain coupling terms $\sigma^{+}_i$ and $\sigma^{-}_i$ are continuous functions defined on $[0, 1]$ with real values.
The boundary coupling coefficients $q_{ij}$, $\rho_{ij}$
are real constants. Finally, we define $\tau_i$ as the total transport time associated with the subsystem $i$: \begin{equation}\label{tau_i def}\tau_i = \frac{1}{\lambda_i} + \frac{1}{\mu_i}.\end{equation}
System~\eqref{system_edp_origin} is represented in Figure~\ref{draw:system_edp}.
 We assume that the initial conditions \( (u_0, v_0) \) belong to \( H^1([0,1], \mathbb{R}^3) \times H^1([0,1], \mathbb{R}^3) \) and satisfy the zero-order compatibility conditions specified in~\citep{BastinCoron2016}.
 Consequently, by applying a folding transformation~\citep{de2020backstepping} and the results from~\citep[Appendix A]{BastinCoron2016}, the original open loop system~\eqref{system_edp_origin} admits a unique solution in $$ C^0([0,+\infty), H^1([0,1], \mathbb{R}^3)) \times C^0([0,+\infty), H^1([0,1], \mathbb{R}^3)). $$ Since the control to be designed will be admissible and continuous in time, the same reference guarantees the well-posedness of the resulting closed-loop system.

In this paper, given a configuration of two active inputs, we want to exponentially stabilize~\eqref{system_edp_origin} in the sense of the $L^2$-norm. More precisely,
we want to find feedback laws such that, in closed loop, there exists $\omega>0$ such that for all $(u_0, v_0)\in H^1([0,1] , \mathbb{R}^3) \times H^1([0,1], \mathbb{R}^3)$ satisfying the zero-order compatibility conditions, the unique solution to \eqref{system_edp_origin} is such that, for all $t\geq 0$,
        $$\|(u(t), v(t))\|_{L^2} \leq K e^{-\omega t}\|(u_0, v_0)\|_{L^2} .$$

\subsection{All two-input configurations}
\label{sec:all_config}
In this subsection, we give all the possible two-input configurations. We deliberately ignore the configurations where one input is at one end of the chain as the stabilization of such configurations has been studied in~\citep{auriol2024output}. 
By a symmetry argument, it is sufficient to restrict the analysis to the stabilization of the four configurations in which exactly two inputs are active (the remaining inputs being set to zero):
\begin{itemize}
    \item\label{config:2}the active inputs are $\Ubar_1$ and $\Ubar_4$; 
    \item\label{config:1} the active inputs are $\Ubar_4$ and $\Ubar_2$; 
    \item\label{config:3}the active inputs are $\Ubar_4$ and $\Ubar_3$;
    \item\label{config:4} the active inputs are $\Ubar_1$ and $\Ubar_3$.
\end{itemize}
In the remainder of the article, we refer to these cases using the terminology \emph{configuration} $(\Ubar_i,\Ubar_j)$.

  \begin{figure}[htb]%

\begin{center}

 \scalebox{1}{

\begin{tikzpicture}



\draw [>=stealth,->,red,very thick] (0,0) -- (3,0);

\draw [red] (1.5,0) node[above]{$u_1(t,x)$};

\draw [>=stealth,<-,blue,very thick] (0,-1.5) -- (3,-1.5);

\draw [blue] (1.5,-1.5) node[below]{$v_1(t,x)$};


\draw [>=stealth,<-,dashed, thick] (1,-1.5) -- (1,0);

\draw(1,-0.75) node[left]{$\sigma^{-}_1$};

\draw [>=stealth,->,dashed, thick] (2,-1.5) -- (2,0);

\draw(2,-0.75) node[left]{$\sigma^{+}_1$};


\draw [blue,>=stealth, thick](-0.5,-0.75) arc (-180:-135:1.1);

\draw [blue,>=stealth,->, thick](-0.5,-0.75) arc (-180:-225:1.1);

\draw [blue] (-0.5,-0.75) node[right]{$q_{11}$};
\draw [>=stealth,<-,purple,very thick] (3.2,-1.6) -- (3.2,-2);

\draw [purple] (3.2,-2) node[below]{$\Ubar_1(t)$};

\draw [red,>=stealth, thick](3.5,-0.75) arc (0:45:1.1);

\draw [red,>=stealth,->, thick](3.5,-0.75) arc (0:-45:1.1);

\draw [red] (3.5,-0.75) node[left]{$\rho_{11}$};


\draw [>=stealth,->,green!50!black!90,very thick] (3.5,0) -- (4.5,0);

\draw [color=green!50!black!90] (4,0) node[above]{$q_{21}$};

\draw [>=stealth,<-,green!50!black!90,very thick] (3.5,-1.5) -- (4.5,-1.5);

\draw [color=green!50!black!90] (4,-1.5) node[below]{$\rho_{12}$};




\draw [>=stealth,->,red,very thick] (5,0) -- (8,0);

\draw [red] (6.5,0) node[above]{$u_{2}(t,x)$};

\draw [>=stealth,<-,blue,very thick] (5,-1.5) -- (8,-1.5);

\draw [blue] (6.5,-1.5) node[below]{$v_{2}(t,x)$};


\draw [>=stealth,<-,dashed, thick] (6,-1.5) -- (6,0);

\draw(6,-0.75) node[left]{$\sigma^{-}_{2}$};

\draw [>=stealth,->,dashed, thick] (7,-1.5) -- (7,0);

\draw(7,-0.75) node[left]{$\sigma^{+}_{2}$};


\draw [blue,>=stealth, thick](4.5,-0.75) arc (-180:-135:1.1);

\draw [blue,>=stealth,->, thick](4.5,-0.75) arc (-180:-225:1.1);
\draw [red] (8.5,-0.75) node[left]{$\rho_{22}$};
\draw [blue] (4.5,-0.75) node[right]{$q_{22}$};

\draw [>=stealth,<-,purple,very thick] (4.9,0.1) -- (4.9,0.5);

\draw [purple] (4.9,0.5) node[above]{$\Ubar_4(t)$};

\draw [red,>=stealth, thick](8.5,-0.75) arc (0:45:1.1);

\draw [red,>=stealth,->, thick](8.5,-0.75) arc (0:-45:1.1);


\draw [>=stealth,<-,purple,very thick] (8.2,-1.6) -- (8.2,-2);

\draw [purple] (8.2,-2) node[below]{$\Ubar_2(t)$};


\draw [>=stealth,->,green!50!black!90,very thick] (8.5,0) -- (9.5,0);

\draw [color=green!50!black!90] (9,0) node[above]{$q_{32}$};

\draw [>=stealth,<-,green!50!black!90,very thick] (8.5,-1.5) -- (9.5,-1.5);

\draw [color=green!50!black!90] (9,-1.5) node[below]{$\rho_{23}$};



\draw [>=stealth,->,red,very thick] (10,0) -- (13,0);

\draw [red] (11.5,0) node[above]{$u_3(t,x)$};

\draw [>=stealth,<-,blue,very thick] (10,-1.5) -- (13,-1.5);

\draw [blue] (11.5,-1.5) node[below]{$v_3(t,x)$};


\draw [>=stealth,<-,dashed, thick] (11,-1.5) -- (11,0);

\draw(11,-0.75) node[left]{$\sigma^{-}_3$};

\draw [>=stealth,->,dashed, thick] (12,-1.5) -- (12,0);

\draw(12,-0.75) node[left]{$\sigma^{+}_3$};


\draw [blue,>=stealth, thick](9.5,-0.75) arc (-180:-135:1.1);

\draw [blue,>=stealth,->, thick](9.5,-0.75) arc (-180:-225:1.1);

\draw [blue] (9.5,-0.75) node[right]{$q_{33}$};

\draw [red,>=stealth, thick](13.5,-0.75) arc (0:45:1.1);

\draw [red,>=stealth,->, thick](13.5,-0.75) arc (0:-45:1.1);

\draw [red] (13.5,-0.75) node[left]{$\rho_{33}$};

\draw [>=stealth,<-,purple,very thick] (9.9,0.1) -- (9.9,0.5);

\draw [purple] (9.9,0.5) node[above]{$\Ubar_3(t)$};


\draw [>=stealth,<->,very thick] (4.5,-3) -- (8.5,-3);

\draw [>=stealth,<->,very thick] (9.5,-3) -- (13.5,-3);

\draw (8,-3) node[above]{1};

\draw (10,-3) node[above]{0};

\draw (5,-3) node[above]{0};

\draw [very thick] (3,-3) -- (3,-3);

\draw (3,-3) node[above]{1};

\draw (13,-3) node[above]{1};

\draw [>=stealth,<->,very thick] (-0.5,-3) -- (3.5,-3);

\draw (0,-3) node[above]{0};


\end{tikzpicture}}
\caption{Schematic representation of system~\eqref{system_edp_origin}}
\label{draw:system_edp}
\end{center}
\end{figure}
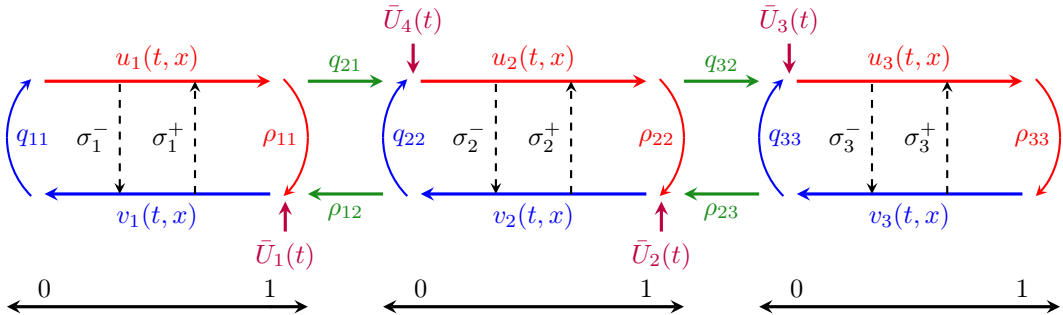

\subsection{Structural Assumptions}
In this section we state the assumptions we use to design our stabilizing control laws.

\begin{assum} \label{Assum_stab_OL}
If $\sigma^{+}_i$ and $\sigma^{-}_i$ are equal to zero for all $i\in \{1,2,3\}$,
then
system~\eqref{system_edp_origin} with all inputs equal to zero is exponentially stable.
\end{assum}
This assumption is necessary to design robust to delays feedback control laws (see~\citep{AuriolHDR} and~\citep{logemann1996conditions}.
As Assumption~\ref{Assum_stab_OL} implies that the boundary coupling terms are dissipative, the presence of the in-domain coupling terms is the main difficulty for the stabilization of~\eqref{system_edp_origin}.
Depending on the configuration, our control design relies on the fact that some of the boundary coupling terms are non-zero, while others may be null. To simplify the presentation, we assume that all of them are non-zero.
\begin{assum} \label{assump:coupling_term_green}
All the coefficients $q_{ij}, \rho_{ij}$ are non-zero.\footnote{\label{footnote:config_assump}One could replace this assumption by asking the following coefficients to be \emph{nonzero}, for \( (\Ubar_1, \Ubar_3): q_{21}, \rho_{23} \), for \( (\Ubar_4, \Ubar_2): q_{32}, \rho_{12} \), for \( (\Ubar_1, \Ubar_4) \) and \((\Ubar_1, \Ubar_2): q_{32}. \)}
\end{assum}
In every configuration except $\Ubar_1, \Ubar_4$, an additional approximate spectral controllability assumption is required for the proposed approach. 
This assumption will be introduced in Section~\ref{Design:feedback}, once the necessary background has been established.

\section{Design of a stabilizing control law in every configuration}
\label{Design:feedback}
For every configuration listed in Section~\ref{sec:all_config}, we show that a stabilizing feedback control law can be designed by reformulating the problem within a framework for which stabilizing strategies are already available.

\subsection[Configuration (U1,U4)]{Configuration $(\Ubar_1,\Ubar_4)$}

\label{section:config1}
Let us define
\[
\Ubar_1(t) = U_1(t)-\rho_{12}\, v_2(t,1), 
\qquad 
\Ubar_4(t) = U_4(t) - q_{21}\, u_1(t,1),
\]
where $U_1, U_4$ denotes new control inputs.  
We now have two independent systems; system 1 actuated by $U_1$ and the chain formed by systems 2 and 3 actuated at one end by $U_4$. The first system can be stabilized using the controller from~\citep{coron2013local}. The second interconnected system can be stabilized by applying the methodology developed in~\citep{auriol2024output}.
These choices fully decouple the first subsystem of~\eqref{system_edp_origin} from the second and third ones.
\subsection{Other configurations}
{In the remaining configurations, we show that the stabilization of the PDE system~\eqref{system_edp_origin} is equivalent to the stabilization of an abstract IDE. 
The latter is achieved under an approximate spectral controllability assumption.}
{Consider the following IDE, \begin{align}
    \label{eq:meta_IDE}
    x(t)&=ax(t-\tau) + \int_0^\tau N(\nu)x(t-\nu)d\nu + b_1V_1(t-\theta_1) \nonumber+\int_0^{\theta_1+\delta_1}M_1(\nu) V_1(t-\nu)d\nu + b_2V_2(t-\theta_2) \\
    &~+\int_0^{\theta_2+\delta_2} M_2(\nu)V_2(t-\nu)d\nu,
\end{align}
with $a\in (-1,1), a\neq0$ $\tau>0$, for all $j\in \{1,2\}$, $b_j\neq0$, $\delta_j\geq0$, $V_j$ is a control input, $\theta_j>0$, $M_j$, $N$ are piecewise continuous functions. The well posedness of~\eqref{eq:meta_IDE} is ensured by~\citep[Proposition~2]{braun2026spectralexponentialstabilitycriterion}, see also~\cite{Fueyo_Baratchart} for manipulation of similar functional equations.
Consider a feedback law of the form
\begin{align}\label{eq:IDE_feedback}
    V_j(t) &= \int_0^{T_0}f^j(\nu) x(t-\nu)d\nu + \int_0^{T_1}g^j_1(\nu)V_1(t-\nu)d\nu+ \int_0^{T_2}g^j_2(\nu)V_2(t-\nu)d\nu,~j\in \{1,2\},
\end{align}
where $f^j,g^j_1,g^j_2$ are piecewise continuous functions and $T_i>0$ for all $i\in\{0,1,2\}$. 
The closed-loop IDE \eqref{eq:meta_IDE}-\eqref{eq:IDE_feedback} is exponentially stable
if there exist
$\omega_0 > 0,$ and $K_0 \geq 1,$ such that for all $x_0 \in L^2(-\tau,0)$, the unique corresponding solution to the IDE~\eqref{eq:meta_IDE} 
    is such that, for all $t\geq0$,    $$\|x_t\|_{L^2(-\tau_,0)} 
    \leq K_0\mathrm{e}^{-\omega_0 t}\|x_0\|_{L^2(-\tau_,0)},$$
    where $x_t(\cdot)$ is the partial trajectory associated to $x$  defined for all $\nu \in [-\tau, 0]$ by $x_t(\nu)= x(t+\nu)$.}
\begin{lem}\label{lem:meta_IDE}
Under Assumptions~\ref{Assum_stab_OL},~\ref{assump:coupling_term_green}, 
 there exists $a\in (-1,1), a\neq0$, $\tau>0$, $\delta_j\geq0$, $b_j\neq0$, $\theta_j>0$, $M_j$, $N$ for all $j\in \{1,2\}$ and a feedback law of the form \eqref{eq:IDE_feedback} such that, the exponential stability of the IDE~\eqref{eq:meta_IDE}-\eqref{eq:IDE_feedback}.
implies the exponential stabilizability of the PDE~\eqref{system_edp_origin} in configurations $(\Ubar_1, \Ubar_3)$, $(\Ubar_4, \Ubar_3)$ and $(\Ubar_4, \Ubar_2)$. 
\end{lem}

The proof of Lemma~\ref{lem:meta_IDE} is postponed to Appendix~\ref{appendix:lem_IDE} for configurations $\Ubar_1, \Ubar_3$ and $\Ubar_4, \Ubar_3$ and to Appendix~\ref{appendix:proof_Lemma_config_U4_U2} for configuration $\Ubar_4, \Ubar_2$. {Our objective is now to stabilize the IDE~\eqref{eq:meta_IDE}.
As $|a|<1$, the principal part of the IDE (i.e., the difference equation $x(t) = ax(t-\tau)$) is exponentially stable~\citep{halebook}. Hence, $\int_0^\tau N(\nu)x(t-\nu)d\nu$ is the main destabilizing term.} The proposed control strategy relies on the following assumption that corresponds to a spectral approximate controllability condition~\citep{fattorini1966some},~\citep{fueyo2025lqapproximatecontrollabilityfrequency}, as in \citep[Chapter~9]{AuriolHDR}. This condition can also be related to an approximated controllability condition for the PDE system~\citep{AuriolSIAM}. 
\begin{assum}[Spectral Assumption]
\label{spectral_assump}
    For all  $s\in \mathbb{C}$,
    $$\operatorname{rank} [F_0(s), F_1(s), F_2(s)] = 1,$$
    where, for all $j\in \{1,2\}$
    \begin{align*}F_0(s) &= 1-ae^{-\tau s} - \int_0^\tau N(\nu)e^{-\nu s}d\nu,\quad F_j(s) = b_je^{-\theta_j s}+\int_0^{\theta_j +\delta_j} M_j(\nu)e^{-\nu s}d\nu.\end{align*}
\end{assum}
Assumption~\ref{spectral_assump} can be weakened to stabilizability by requiring full rank only on the shifted right half-plane. Under this assumption, the stabilization of a general IDE was recently achieved in~\cite{braun_corona}.
In Section~\ref{contre_exemple}, we will construct an example of a configuration where this assumption is not satisfied.
{
In the configuration in which only $\Ubar_1$ acts as an input for the PDE system~\eqref{system_edp_origin}, a stabilizing controller has been designed in~\citep{braun2025stabilizationchainhyperbolicpdes} under the stronger spectral controllability condition, for all $s\in \C$,
$\operatorname{rank}[F_0(s), F_1(s)] =1.$
}
\begin{thm}
    \label{thm:main_stab}
    Under Assumptions~\ref{Assum_stab_OL},~\ref{assump:coupling_term_green}, and,~\ref{spectral_assump} 
there exist state-feedback control laws that ensure the exponential stability of the PDE system~\eqref{system_edp_origin} in configurations $(\Ubar_1, \Ubar_3)$, $(\Ubar_4, \Ubar_3)$ and $(\Ubar_4, \Ubar_2)$. 
\end{thm}
  
\begin{proof}
According to Lemma~\ref{lem:meta_IDE}, we just have to exponentially stabilize the IDE~\eqref{eq:meta_IDE} 
Our goal will be to rewrite the system in the framework proposed in~\citep{AuriolLCSS}.
    Taking the Laplace transform of the IDE~\eqref{eq:meta_IDE}, we obtain for all $s\in \C$,
    $$F_0(s)x(s) = F_1(s)V_1(s) + F_2(s)V_2(s),$$
where $x(s)$, $V_j(s)$ are the Laplace transforms of $x$, respectively $V_j$.
 As $b_j\neq 0$ for all $j\in \{1,2\}$, $F_0, F_1, F_2$ are not identically equal to zero.
From this point, we treat configurations $\Ubar_1,\Ubar_3$ and $\Ubar_4,\Ubar_3$ separately from configuration $\Ubar_4,\Ubar_2$.
\subsubsection[Configurations 1,3 et 4,3]{Configurations $(\Ubar_1,\Ubar_3)$ and $(\Ubar_4,\Ubar_3)$.} The following proof holds for configurations $\Ubar_1,\Ubar_3$ and $\Ubar_4,\Ubar_3$. Recall that in these configurations, $\delta_1=0$ in the IDE~\eqref{eq:meta_IDE}.
By Appendix~\ref{proof:lem_spect}, there exist $\alpha\in \R$ and $T\in (0, \tau)$ such that, for all $s\in \C$
\begin{equation}\operatorname{rank}\Big[F_0(s) -\alpha \int_0^T e^{-\nu s}d\nu F_2(s),F_1(s)\Big]=1,\label{spectral_controllability_lcss}\end{equation} 
Then, choosing, 
\begin{equation}V_2(t):= \alpha \int_0^T x(t-\nu)d\nu, \label{def:W2}\end{equation}
the IDE~\eqref{eq:meta_IDE} becomes (in Laplace space),
\begin{equation}\label{eq:IDE_x_U1}\Big(F_0(s)-\alpha \int_0^T e^{-\nu s}d\nu F_2(s)\Big)x(s) = F_1(s)V_1(s).\end{equation}
which equivalently rewrites, for all $t>\tau + \theta_2 +\delta_2+\theta_1$,
\begin{align}x(t) &= ax(t-\tau) + \int_0^{\tau +\theta_2 +\delta_2}\Tilde{N}(\nu)x(t-\nu)d\nu + b_1V_1(t-\theta_1)+\int_0^{\theta_1}M_1(\nu) V_1(t-\nu), \label{IDE_x_W1_temp}\end{align}
with $\Tilde{N}$ a piecewise continuous function.
{Because of Condition~\eqref{spectral_controllability_lcss}, the IDE~\eqref{IDE_x_W1_temp} falls within the framework of~\citep{AuriolLCSS}, the only difference being that the distributed delay exceeds the pointwise delay appearing in the state. 
However, the method proposed in~\citep{AuriolLCSS} extends straightforwardly.\footnote{This would not be the case for configuration $\Ubar_4, \Ubar_2$, hence the dissociation.} 
}
More precisely, we can design an auto-regressive stabilizing feedback  control law, defined as
\begin{equation} \label{eq:W1}V_1(t) = \int_0^{\tau+\theta_2+\delta_2} g(\nu)x(t-\nu)d\nu + \int_0^{\theta_1}f(\nu)V_1(t-\nu),\end{equation}
where the gains $g$ and $f$ are piecewise continuously differentiable. As shown in~\citep{AuriolLCSS}, they can be computed by solving a set of Fredholm equations. 
\subsubsection[Config 4,2]{Configuration $(\Ubar_4, \Ubar_2)$.} 
 Due to Appendix~\ref{proof:lem_spect}, there exist $\alpha\in \R$ and $T\in (0, \theta_1 + \delta_1)$ such that, for all $s\in \C$
\begin{equation}\operatorname{rank}\Big[F_0(s) ,F_1(s)-\alpha \int_0^{T}e^{-\nu s}d\nu F_2(s)\Big]=1,\label{spectral_controllability_lcss2}\end{equation} 
Then, choosing, 
\begin{equation}V_2(t):= -\alpha \int_0^{T} V_1(t-\nu)d\nu, \label{def:V2}\end{equation}
the IDE~\eqref{eq:meta_IDE} becomes (in Laplace space),
\begin{equation}\label{eq:IDE_x_V1}F_0(s)x(s) = \Big(F_1(s)-\alpha \int_0^{T} e^{-\nu s}d\nu F_2(s)\Big)V_1(s).\end{equation}
{which equivalently rewrites, for all $t>\tau + \theta_2 +\theta_1 +\delta_1 + \delta_2$,
\begin{align}x(t) &= ax(t-\tau) + \int_0^{\tau}N(\nu)x(t-\nu)d\nu+ b_1V_1(t-\theta_1)+\int_0^{\theta_1+\theta_2+\delta_1+\delta_2}\Tilde{M}_1(\nu) V_1(t-\nu), \label{IDE_x_V1_temp}\end{align}
with $\Tilde{M}$ a piecewise continuous function.}
{Due to Condition~\eqref{spectral_controllability_lcss2}, the IDE~\eqref{IDE_x_V1_temp} falls within the framework of~\citep{AuriolLCSS}, the only difference being that the distributed delay exceeds the pointwise delay appearing in the input. 
In this setting, the method introduced in~\citep{AuriolLCSS} extends straightforwardly.
}
More precisely, we can design an auto-regressive stabilizing feedback  control law, defined as
\begin{equation} \label{eq:V1}V_1(t) = \int_0^{\tau} g(\nu)x(t-\nu)d\nu + \int_0^{\theta_1+\theta_2+\delta_1+\delta_2}f(\nu)V_1(t-\nu),\end{equation}
where the gains $g$ and $f$ are piecewise continuously differentiable. As shown in~\citep{AuriolLCSS}, they can be computed by solving a set of Fredholm equations. 
This concludes the proof.
\end{proof}
\section{A nontrivial example where the Spectral Assumption is not satisfied}
\label{contre_exemple}
In this section we construct a nontrivial example where the IDE formulation does not satisfy Assumption~\ref{spectral_assump}. {This configuration constitutes a non-degenerate case for which our approach does not apply.}
In system~\eqref{system_edp_origin}, we choose Configuration~$(\bar U_1,\bar U_3)$, 
we take $\lambda_i>0$, $\mu_i >0$,
$\sigma_1^-=\sigma_2^+=\sigma_3^+=0$.
 The degrees of freedom available to us are the choice of 
$\sigma_1^+$, $\sigma_2^-$, $\sigma_3^-$, together with the choice of the boundary coupling coefficients $q_{ij}$ and $\rho_{ij}$. 
The coefficients $q_{21}$ and $\rho_{23}$ must be nonzero in order for the inputs to act on the second subsystem 
and all coefficients are selected so that Assumption~\ref{Assum_stab_OL} holds. We set $q_{32} = q_{33} = \rho_{11} = \rho_{12} =0$, hence Assumption~\eqref{assump:coupling_term_green} will not be satisfied. However this is not a problem as it was a non necessary assumption used to simplify the exposition (see footnote~\ref{footnote:config_assump}).
Using the method of characteristics, we will rewrite our problem~\eqref{system_edp_origin} as a IDE (no backstepping is needed in this special decoupled case). Then, we will choose the remaining boundary coupling coefficients $q_{ij}, \rho_{ij}$ and the coupling terms $\sigma_1^+,\sigma_2^-,\sigma_3^-$ such that Assumption~\ref{spectral_assump} is not satisfied.
Set $x(t):=u_2(t,0)$, using the method of characteristics, $x$ is solution of, for all $t>\tau_1+\tau_2+\tau_3$,
\begin{equation}
    \label{eq:IDE_c_e}
    \begin{aligned}
        x(t)
    &=q_{22}\rho_{23}\rho_{33}U_3(t-(\tau_3+1/\mu_2))+q_{22}\rho_{23}\sigma_3^-\frac{1}{1+\frac{\mu_3}{\lambda_3}}\int_{\frac{1}{\mu_2}}^{\frac{1}{\mu_2}+\tau_3}U_3(t-\nu)d\nu +q_{22}\rho_{22}x(t-\tau_2) \\&~+q_{22}\sigma_2^-\frac{1}{1+\frac{\mu_2}{\lambda_2}}\int_0^{\tau_2}x(t-\nu)d\nu+q_{21}q_{11}U_1(t-\tau_1) + q_{21}\sigma_1^+\frac{1}{1+\frac{\lambda_1}{\mu_1}}\int_0^{\tau_1} U_1(t-\nu)d\nu.
    \end{aligned}
\end{equation}
Taking the Laplace transform of~\eqref{eq:IDE_c_e}, we obtain for all $z \in \C$
$$F_0(z)x(z) = F_2(z)U_3(z) + F_1(z)U_1(z),$$
with 
$$F_0(z) = 1-q_{22}\rho_{22}e^{-\tau_2 z}-q_{22}\sigma_2^-\frac{1}{1+\frac{\mu_2}{\lambda_2}}\int_0^{\tau_2}e^{-\nu z}d\nu,$$
$$F_1(z) = q_{21}q_{11}e^{-\tau_1z}+ q_{21}\sigma_1^+\frac{1}{1+\frac{\lambda_1}{\mu_1}}\int_0^{\tau_1}e^{-\nu z}d\nu,$$
and,
\begin{align*}F_2(z) &= q_{22}\rho_{23}\rho_{33}e^{-(\tau_3+1/\mu_2)z}+q_{22}\rho_{23}\sigma_3^-\frac{1}{1+\frac{\mu_3}{\lambda_3}}\int_{\frac{1}{\mu_2}}^{\frac{1}{\mu_2}+\tau_3}e^{-\nu z}d\nu.\end{align*}
Now fix, \begin{equation}\label{coeff:ce_q_rho}
\rho_{22} =1/2, \quad q_{22}=\rho_{23}=\rho_{33}=q_{21}=q_{11}=1,\end{equation}
and choose the in domain couplings terms to normalize the integral terms, i.e,
\begin{equation}
    \label{coupling:ce_sigma}
    \sigma_2^- = \frac{1+\frac{\mu_2}{\lambda_2}}{2\tau_2}, \quad
    \sigma_1^+=-\frac{1+\frac{\lambda_1}{\mu_1}}{\tau_1}, \quad
    \sigma_3^{-}=-\frac{1+\frac{\mu_3}{\lambda_3}}{\tau_3}.
\end{equation}
Then, the principal part of the IDE~\eqref{eq:IDE_c_e} (i.e the IDE without the integral terms and with $U_1 =U_2 \equiv0$) is exponentially stable because $|q_{22}\rho_{22}|<1$ (see~\cite{halebook}), and,
$$F_0(0) = F_1(0)=F_2(0)=0.$$
We may therefore state the following theorem.
\begin{thm}
\label{thm:contre_exemple}
Let $\lambda_i > 0$ and $\mu_i > 0$ be arbitrary transport speeds, for all $i \in \{1,2,3\}$. 
Then there exist coupling coefficients $q_{ij}$ and $\rho_{ij}$, as well as constant boundary coupling terms $\sigma_i^{+}$ and $\sigma_i^{-}$, such that the following holds. 
In configuration $\Ubar_1, \Ubar_3$,  Assumption~\ref{Assum_stab_OL} and footnote~\ref{footnote:config_assump} stand.
However, the IDE reformulation~\eqref{eq:IDE_c_e} of the original system~\eqref{system_edp_origin} does not satisfy Assumption~\ref{spectral_assump}.
\end{thm}

\section{Conclusion}

In this paper, we have addressed the stabilization of a chain of three hyperbolic PDEs for all two-input boundary configurations. 
The methodology builds on the framework introduced in~\citep[Chapter~9]{AuriolHDR} and extended in~\citep{braun2025stabilizationchainhyperbolicpdes}. 
Because the control design relies on the cancellation of boundary reflection terms, the resulting feedback law may lack robustness to input delays. However, robustness properties can be recovered by applying an appropriately tuned low-pass filter, as shown in~\citep{auriol2023robustification}.

With two inputs, stabilization of the three-PDE chain is always achievable when one input is applied at an endpoint of the chain or when two inputs fully control a single node, as in configuration $(\Ubar_1,\Ubar_4)$. 
For the remaining configurations, an additional spectral assumption~\ref{spectral_assump} is required to apply the proposed approach. 
This assumption, which asserts that three associated holomorphic functions do not vanish simultaneously, is weaker than the spectral controllability condition used in the single-input case~\citep{braun2025stabilizationchainhyperbolicpdes}. 
Nevertheless, as demonstrated in the final section, there exist configurations in which this assumption fails, thereby revealing intrinsic limitations of the methodology.

Future work will focus on extending these results to the stabilization of a chain composed on $N$ systems under arbitrary configurations of $k<N$ boundary inputs. It would also be of great interest to compare the control effort and closed-loop performance of the proposed stabilizing controllers in each configuration.

                                                   







\appendix
\section{Proof of Lemma~\ref{lem:meta_IDE}}
\label{appendix:lem_IDE}
\begin{proof}(for Configuration $\Ubar_1, \Ubar_3$).

We use the \emph{invertible} backstepping transformation $\Bar{\mathcal{F}}$ that maps $(u_i, v_i)_{1\leq i\leq 3}\in H^1([0,1],\mathbb{R}^6)$ to $(\alpha_i, 
\beta_i)_{1\leq i\leq 3}\in H^1([0,1],\mathbb{R}^6)$
defined by 
\begin{equation}
\begin{aligned}
    \mathcal{\alpha}_1(x) &= u_1(x)
    -\int_0^x \begin{pmatrix}
        K_1^{uu}(x, y)\\ K_1^{uv} (x, y)
    \end{pmatrix}^\top\begin{pmatrix}
        u_1(y)\\ v_1(y) 
    \end{pmatrix}dy 
   ,\\
    \beta_1(x) &= v_1(x)  -\int_0^x \begin{pmatrix}
        K_1^{vu}(x, y)\\ K_1^{vv} (x, y)
    \end{pmatrix}^\top\begin{pmatrix}
        u_1(y)\\ v_1(y) 
    \end{pmatrix}dy,
\\
    \mathcal{\alpha}_2(x) &= u_2(x) -\int_x^1 \begin{pmatrix}
        K_2^{uu}(x, y)\\ K_2^{uv} (x, y)
    \end{pmatrix}^\top\begin{pmatrix}
        u_2(y)\\ v_2(y) 
    \end{pmatrix}dy, \\
    \beta_2(x) &= v_2(x)  -\int_x^1 \begin{pmatrix}
        K_2^{vu}(x, y)\\ K_2^{vv} (x, y)
    \end{pmatrix}^\top\begin{pmatrix}
        u_2(y)\\ v_2(y) 
    \end{pmatrix}dy,
\\ 
    \mathcal{\alpha}_3(x) &= u_3(x) -\int_x^1 \begin{pmatrix}
        K_3^{uu}(x, y)\\ K_3^{uv} (x, y)
    \end{pmatrix}^\top\begin{pmatrix}
        u_3(y)\\ v_3(y) 
    \end{pmatrix}dy,\\
    \beta_3(x) &= v_3(x)  -\int_x^1 \begin{pmatrix}
        K_3^{vu}(x, y)\\ K_3^{vv} (x, y)
    \end{pmatrix}^\top\begin{pmatrix}
        u_3(y)\\ v_3(y) 
    \end{pmatrix}dy, 
\end{aligned} \label{backstep_Fbar}
\end{equation}
where $K_i^{uu}$, $K_i^{uv}$, $K_i^{vu}$, $K_i^{vv}$ are piecewise continuous kernels defined {the same way as in~\citep{braun2025stabilizationchainhyperbolicpdes}} and ${}^\top$ denotes the transpose operator.
The transformation $\Bar{\mathcal{F}}$ maps the original system \eqref{system_edp_origin} (with $\Ubar_2=\Ubar_4\equiv 0$) into the following system: $\forall i\in \{1,3\}$, for all $x\in [0,1]$,
\begin{equation} \label{interm:target_config_U1_U_3}
\begin{aligned}
    \partial_t \alpha_2(t,x) + \lambda_2 \partial_x \alpha_2(t,x) &=-\mu_2K_2^{uv}(x,1)\rho_{23}v_3(t,0), \\
    \partial_t \beta_2(t,x) - \mu_2 \partial_x \beta_2(t,x) &=-\mu_2K_2^{vv}(x,1)\rho_{23}v_3(t,0),\\
      \partial_t \alpha_i(t,x) + \lambda_i \partial_x \alpha_i(t,x) &= 0, \\
    \partial_t \beta_i(t,x) - \mu_i \partial_x \beta_i(t,x) &=0.
\end{aligned}
\end{equation}
The map $\Bar{\mathcal{F}}$ is known to be continuously invertible as a Volterra transform (see, e.g., \citep{Yoshida1960}).
Moreover, {using the inverse transform $\Bar{\mathcal{F}}^{-1}$ (that has the same structure as $\Bar{\mathcal{F}}$} we obtain the boundary conditions  
\begin{equation}
\begin{aligned}
 \alpha_1(t,0) &= q_{11}\beta_1(t,0),\\
 \beta_1(t,1) &=U_1(t)\\
\alpha_2(t,0) &= q_{22}\beta_2(t,0) + q_{21}\alpha_1(t,1)  \int_0^1 \sum_{i=1}^2 \alpha_i(t,y) P_{2i}(y)+\beta_i(t,y)W_{2i}dy\\
 \beta_2(t,1) &=\rho_{22}\alpha_2(t,1) + \rho_{23}\beta_3(t,0)+ \int_0^1 \alpha_3(t,y) J_{23}(y)+\beta_3(t,y)C_{23}dy, \\
 \alpha_3(t,0)&= U_3(t),
 \\
 \beta_3(t,1) &= \rho_{33}\alpha_3(t,1),
\end{aligned}\label{boundary_cond_target_system:config_U1_U3}
\end{equation}
where we chose $\Ubar_1$ and $\Ubar_3$ such that,
\begin{align*}
    \Ubar_1(t)&= U_1(t)-\rho_{12}\beta_2(t,0) - \rho_{11}\alpha_1(t,1)-\int_0^1 \sum_{i=1}^2 \alpha_i(t,y) J_{1i}(y)+\beta_i(t,y)C_{1i}dy,\\
 \Ubar_3(t)&=U_3(t)-  q_{32} \alpha_2(t,1) - q_{33}\beta_3(t,0) - \int_0^1 \alpha_3(t,y)P_{33}(y) + \beta_3(t,y)W_{33}(y)dy,
\end{align*}
with $U_1$ and $U_3$ being the new control inputs.
The kernels $J_{ij}, C_{ij}, P_{ij}, W_{ij}$ are piecewise continuous functions defined using the inverse transform $\Bar{\mathcal{F}}^{-1}$
Set $x(t):= \alpha_2(t,1)$, using the method of characteristics by following the method introduced in~\citep[Chapter 9]{AuriolHDR} and used in~\citep{braun2025stabilizationchainhyperbolicpdes}, we show that $x$ is solution to the following IDE, defined for all $t>\tau_1 + \tau_2+\tau_3$ by
\begin{align}\label{IDE:config_U1_U_3} \nonumber
     x(t) &= q_{22}\rho_{22}x(t-\tau_2)+q_{22}\rho_{23}\rho_{33}U_3(t-(\tau_3+1/\mu_2)) +q_{21}q_{11}U_1(t-\tau_1)\\
     &~+ \int_0 
^{\tau_1} U_1(t-\nu)M_1(\nu)d\nu +\int_0
^{\tau_3+\frac{1}{\mu_2}} U_3(t-\nu)M_3(\nu)d\nu +\int_0
^{\tau_2} x(t-\nu)N(\nu)d\nu,
\end{align}
where $M_1, N, M_3$ are piecewise continuous functions depending on the kernels of the inverse transform $\Bar{\mathcal{F}}^{-1}$.
Using~\citep[Theorem~6.1.3]{AuriolHDR}, the exponential stability of the PDE system~\eqref{system_edp_origin} in configuration $\Ubar_1, \Ubar_3$ is equivalent to the exponential stability of the IDE~\eqref{IDE:config_U1_U_3}. 
Moreover, as shown in~\citep{braun2025stabilizationchainhyperbolicpdes}, Assumption~\ref{Assum_stab_OL}, together with the method of characteristics and \citep[Theorem~6.1.3]{AuriolHDR}, ensures that $|q_{11}\rho_{11}| < 1$. Hence, the principal part of the IDE~\eqref{IDE:config_U1_U_3} is exponentially stable (see~\citep{halebook}).
In this configuration, under Assumption~\ref{assump:coupling_term_green} and with $V_1 = U_1$, $V_2= U_3$ the IDE~\eqref{IDE:config_U1_U_3} takes the same form as~\eqref{eq:meta_IDE}.
\end{proof}
\begin{proof}[Configuration $\Ubar_4, \Ubar_3$]
     The proof follows the same structure as in the previous configuration. 
We apply the same backstepping transformation $\Bar{\mathcal{F}}$ to map the PDE system~\eqref{system_edp_origin} in configuration $(\Ubar_4,\Ubar_3)$ to the target system~\eqref{boundary_cond_target_system:config_U1_U3}, with only minor changes in the boundary conditions. 

Setting $x(t) := \beta_1(t,1)$ and proceeding as in the previous configuration, we obtain that the exponential stability of the PDE system~\eqref{system_edp_origin} is equivalent to the exponential stability of the IDE
\begin{align}
\label{IDE:config_U4_U3}
x(t) &= q_{11}\rho_{11} \, x(t-\tau_1) + \rho_{22}\rho_{12} \, U_4(t-\tau_2) \nonumber+ \rho_{33}\rho_{23}\rho_{12} \, U_3(t-\tau_3 - 1/\mu_2)\\
&~+ \int_0^{\tau_2} M_4(\nu) \, U_4(t-\nu)\, d\nu
      + \int_0^{\tau_2+\tau_3} M_3(\nu) \, U_3(t-\nu)\, d\nu  + \int_0^{\tau_1} N(\nu) \, x(t-\nu)\, d\nu,
\end{align}
valid for all $t > \tau_1 + \tau_2 + \tau_3$.
Here, $M_4$, $M_3$, and $N$ are piecewise continuous functions.  
Moreover, as shown in~\citep{braun2025stabilizationchainhyperbolicpdes}, Assumption~\ref{Assum_stab_OL}, together with the method of characteristics and \citep[Theorem~6.1.3]{AuriolHDR}, ensures that $|q_{11}\rho_{11}| < 1$. Hence, the principal part of the IDE~\eqref{IDE:config_U4_U3} is exponentially stable (see~\citep{halebook}).

Finally, under Assumption~\ref{assump:coupling_term_green}, and with the identifications
\[
V_2 := U_3, 
\qquad 
V_1 := U_4,
\]
the IDE~\eqref{IDE:config_U4_U3} is of the same form as the abstract IDE~\eqref{eq:meta_IDE}. 
This completes the proof.
\end{proof}
\section{Proof of Lemma~\ref{lem:meta_IDE} second part}
\label{appendix:proof_Lemma_config_U4_U2}
\begin{proof}[Configuration $\Ubar_4, \Ubar_2$]
We use the \emph{invertible} backstepping transformation $\mathcal{F}$ defined in~\citep{braun2025stabilizationchainhyperbolicpdes}, to map the original system \eqref{system_edp_origin} into the following target system: for all $i\in \{1,3\}$, for all $x\in [0,1]$,
\begin{equation}\label{interm:target_config_U4_U2}
\begin{aligned}
    \partial_t \alpha_i(t,x) + \lambda_i \partial_x \alpha_i(t,x) &= 0, \\
    \partial_t \beta_i(t,x) - \mu_i \partial_x \beta_i(t,x) &=0, \\
    \partial_t \alpha_2(t,x) + \lambda_2 \partial_x \alpha_2(t,x) &=-\mu_2K_2^{uv}(x,1)\Ubar_2(t), \\
    \partial_t \beta_2(t,x) - \mu_2 \partial_x \beta_2(t,x) &=-\mu_2K_2^{vv}(x,1)\Ubar_2(t).
\end{aligned}
\quad 
\end{equation}
The functions $K_i^{uu}$, $K_i^{uv}$, $K_i^{vu}$, $K_i^{vv}$, $Q^{\alpha}$, $Q^{\beta}$, $R^{\alpha}$ are the kernels of the backstepping transform $\mathcal{F}$
 defined in~\citep{braun2025stabilizationchainhyperbolicpdes}.
Originally, the kernels $Q^{\alpha}$, $Q^{\beta}$, $R^{\alpha}$ and $R^{\beta}$ were defined in~\citep{braun2025stabilizationchainhyperbolicpdes} with an added artificial boundary condition necessary for well-posedness, namely for all $y\in [0,1]$,
$Q^{\alpha}(1,y)=Q^{\beta}(1,y) =R^{\alpha}(1,y) =R^{\beta}(1,y)=0$. These boundary conditions can be freely chosen (see~\citep[Lemma 8]{AuriolPietri}); in this context, because $q_{32}\neq0$ (under Assumption~\ref{assump:coupling_term_green}) we can choose,
\begin{align*}R^\alpha(1,y) &=\frac{1}{q_{32}}\big(K_3^{uv}(0,y)-q_{33}K_3^{vv}(0,y)\big),\\
Q^\alpha(1,y) &=\frac{1}{q_{32}}\big(K_3^{uu}(0,y)-q_{33}K_3^{vu}(0,y)\big),\\
R^\beta(1,y)& = \rho_{22}R^\alpha(1,y) +\rho_{23}K_3 ^{vv}(0,y),\\
Q^\beta(1,y) &= \rho_{22}Q^\alpha(1,y) + \rho_{23}K_3^{vu}(0,y). \end{align*}{This choice enables us to cancel the integral terms that could appear at the boundaries $\alpha_3(t,0)$ and $\beta_2(t,1)$.}
Hence, {using the inverse transform $\mathcal{F}^{-1}$ (that has the same structure as $\mathcal{F}$}) we obtain the boundary conditions  
\begin{equation}
\begin{aligned}
 \alpha_1(t,0) &= q_{11}\beta_1(t,0),\\
 \beta_1(t,1) &= \rho_{12}\beta_2(t,0) + \rho_{11}\alpha_1(t,1)+\int_0^1 \sum_{i=1}^3 \alpha_i(t,y) J_{1i}(y)+\beta_i(t,y)C_{1i}dy,\\
\alpha_2(t,0) &=U_4(t) \\
 \beta_2(t,1) &= \Ubar_2(t) +\rho_{22}\alpha_2(t,1) + \rho_{23}\beta_3(t,0)\\
 \alpha_3(t,0)&= q_{32} \alpha_2(t,1) + q_{33}\beta_3(t,0)
 \\
 \beta_3(t,1) &= \rho_{33}\alpha_3(t,1),
\end{aligned}\label{boundary_cond_target_system_config_U4_U2}
\end{equation}
where we chose $\Ubar_4$ such that,
\begin{align*}\Ubar_4(t) &= U_4(t) -q_{22}\beta_2(t,0) - q_{21}\alpha_1(t,1) - \int_0^1 \Big(\sum_{i=1}^3 \alpha_i(t,y) P_{2i}(y)+\beta_i(t,y)W_{2i}\big)dy,\end{align*}
where $U_4$ is a new input.
The kernels $J_{ij}, C_{ij}, P_{ij}, W_{ij}$ are piecewise continuous functions defined in~\citep{braun2025stabilizationchainhyperbolicpdes} that depend on the coupling terms $\sigma_i^+, \sigma_i^-$.
In order to exponentially stabilize the PDE system~\eqref{system_edp_origin} in this configuration, we propose to rewrite the intermediate target system~\eqref{interm:target_config_U4_U2},~\eqref{boundary_cond_target_system_config_U4_U2} as two IDEs.
 Set $x_3(t):= \alpha_3(t,0)$ and $x_1(t) = \beta_1(t,1)$, using the method of characteristics and following the method proposed in~\citep{AuriolHDR} and used in~\citep{braun2025stabilizationchainhyperbolicpdes}, we show that, for all $t>\tau_1 + \tau_2+\tau_3$,
 \begin{align}
     \label{IDE_x3_config_U4_U2}
     x_3(t)&=q_{33}\rho_{33}x_3(t-\tau_3) + q_{32}U_4(t-1/\lambda_2)+\int_0^{\frac{1}{\lambda_2}}M_{32}(\nu)\Ubar_2(t-\nu)d\nu,
 \end{align}
 \begin{align}
     \label{IDE_x1_config_U4_U2}
     x_1(t)&= \rho_{11}q_{11}x_1(t-\tau_1) +\int_0^{\tau_1}
N_1(\nu)x_1(t-\nu)d\nu +\rho_{33}\rho_{23}\rho_{12}x_3(t-\tau_3-1/\mu_2) \nonumber\\
&~+\int_0^{\tau_3+\frac{1}{\mu_2}}N_{13}(\nu)x_3(t-\nu)d\nu+\rho_{22}\rho_{12}U_4(t-\tau_2) +\int_0^{\tau_2}M_{14}(\nu)U_4(t-\nu)d\nu\nonumber\\
&~+\rho_{12}\Ubar_2(t-1/\mu_2) +\int_0^{\tau_2}M_{12}(\nu)\Ubar_2(t-\nu)d\nu,\end{align}
where $M_{34}, N_1, N_{13}, M_{14}, M_{12}$ are piecewise continuous functions.
Using~\citep[Theorem~6.1.3]{AuriolHDR}, the exponential stability of the PDE system~\eqref{system_edp_origin} in configuration $\Ubar_4, \Ubar_3$ is equivalent to the exponential stability of the IDEs~\eqref{IDE_x1_config_U4_U2},\eqref{IDE_x3_config_U4_U2}. Furthermore, using Assumption~\ref{Assum_stab_OL},~\citep[Theorem~6.1.3]{AuriolHDR}, and the method of characteristics, we have $|q_{33}\rho_{33}|<1$ and $|\rho_{11}q_{11}|<1$. Hence the difference equations 
\begin{align}  x_3(t)&=q_{33}\rho_{33}x_3(t-\tau_3)\label{ppx3}\\
 x_1(t)&= \rho_{11}q_{11}x_1(t-\tau_1) \label{ppx1}
\end{align}
are exponentially stable~\citep{halebook}.
{From this point onward, the objective is to reduce the stabilizability analysis of the IDEs~\eqref{IDE_x3_config_U4_U2} and~\eqref{IDE_x1_config_U4_U2} to the analysis of a single IDE of the form~\eqref{eq:meta_IDE}. 
The strategy is to first eliminate the integral term in~\eqref{IDE_x3_config_U4_U2}, and then remove the pointwise delayed $U_4$ term in~\eqref{IDE_x1_config_U4_U2}. Finally we will inject~\eqref{IDE_x3_config_U4_U2} in~\eqref{IDE_x1_config_U4_U2}.
}
To do so, we apply the following change of variable,
 \begin{equation} \label{changevarx3bar}
\Bar{x}_3(t) = x_3(t) - \int_0^{\frac{1}{\lambda_2}}M_{32}(\nu)\Ubar_2(t-\nu)d\nu.
\end{equation}
Hence the IDE~\eqref{IDE_x3_config_U4_U2} becomes,
\begin{align*}
    \Bar{x}_3(t) &= q_{33}\rho_{33}\Bar{x}_3(t-\tau_3) + U_4(t-1/\lambda_2)+  q_{33}\rho_{33}\int_0^{\frac{1}{\lambda_2}}M_{32}(\nu)\Ubar_2(t-\nu-\tau_3)d\nu
\end{align*}

Let us assume here that $\tau_3>1/\lambda_2$ without loss of generality. If $\tau_3\leq \frac{1}{\lambda_2}$ then one should first iterate $n$ times the change of variable~\eqref{changevarx3bar}, 
with $n\in \N$ such that $n\tau_3>1/\lambda_2$.
Hence, we can choose $U_4$ such that,
$$U_4(t) = \Tilde{U}_4(t) - \frac{ q_{33}\rho_{33}}{q_{32}}\int_0^{\frac{1}{\lambda_2}}M_{32}(\nu) \Ubar_2(t-\nu-\tau_3+1/\lambda_2)d\nu,$$
with $\Tilde{U}_4$ being the new control input.
We obtain the following IDE,
\begin{align}\label{eq:IDEx3bar}
    \Bar{x}_3(t) &= q_{33}\rho_{33}\Bar{x}_3(t-\tau_3) + q_{32}\Tilde{U}_4(t-1/\lambda_2)
\end{align}
Moreover, the IDE~\eqref{IDE_x1_config_U4_U2} becomes,
\begin{align}
    \label{eqintermediaire1}
     x_1(t)&= \rho_{11}q_{11}x_1(t-\tau_1) +\int_0^{\tau_1}
N_1(\nu)x_1(t-\nu)d\nu \nonumber+\rho_{33}\rho_{23}\rho_{12}\Bar{x}_3(t-\tau_3-1/\mu_2)\\
&~+\int_0^{\tau_3+\frac{1}{\mu_2}}N_{13}(\nu)\Bar{x}_3(t-\nu)d\nu+\rho_{22}\rho_{12}\Tilde{U}_4(t-\tau_2) +\int_0^{\tau_2}M_{14}(\nu)\Tilde{U}_4(t-\nu)d\nu\nonumber\\
&~+\rho_{12}\Ubar_2(t-1/\mu_2) +\int_0^{{\tau_2 +\tau_3}}\Tilde{M}_{12}(\nu)\Ubar_2(t-\nu)d\nu,
\end{align}
with $\Tilde{M}_{12}$ a piecewise continuous function.
Then, we choose,
$$\Ubar_2(t) = U_2(t) - \rho_{22}\Tilde{U}_4(t-1/\lambda_2),$$
with $U_2$ being a new control input.
Hence, the IDE~\eqref{eqintermediaire1} becomes,
\begin{align}
    \label{eqintermediaire2}
     x_1(t)&= \rho_{11}q_{11}x_1(t-\tau_1) +\int_0^{\tau_1}
N_1(\nu)x_1(t-\nu)d\nu +\rho_{33}\rho_{23}\rho_{12}\Bar{x}_3(t-\tau_3-1/\mu_2)\nonumber \\
&~+\int_0^{\tau_3+\frac{1}{\mu_2}}N_{13}(\nu)\Bar{x}_3(t-\nu)d\nu+\int_0^{\tau_2+\tau_3+\frac{1}{\lambda_2}}\Tilde{M}_{14}(\nu)\Tilde{U}_4(t-\nu)d\nu\nonumber\\
&~+\rho_{12}U_2(t-1/\mu_2) +\int_0^{{\tau_2 +\tau_3}}\Tilde{M}_{12}(\nu)U_2(t-\nu)d\nu,
\end{align}
with $\Tilde{M}_{14}$ a piecewise continuous function.
{Then, defining the new input,
$$\Bar{\Bar{U}}_4(t)=\Tilde{U}_4(t) -q_{33}\rho_{33}\Tilde{U}_4(t-\tau_3),$$}
we obtain
for all $t>2\tau_2+\tau_3+\tau_1+1/\lambda_2$,
    \begin{align}
    \label{IDE_finale_config_U4_U3}
     x(t)&= \rho_{11}q_{11}x(t-\tau_1) +\int_0^{\tau_1}
N_1(\nu)x(t-\nu)d\nu +\rho_{33}\rho_{23}\rho_{12}q_{32}\Tilde{U}_4(t-\tau_3-\tau_2)\nonumber \\
&~+\int_0^{2\tau_3+\tau_2+\frac{1}{\lambda_2}}\Tilde{M}_4(\nu)\Bar{\Bar{U}}_4(t-\nu)d\nu+\rho_{12}\Tilde{U}_2(t-1/\mu_2) \nonumber\\
&~+\int_0^{{\tau_2 +\tau_3}}\Tilde{M}_{12}(\nu)\Tilde{U}_2(t-\nu)d\nu,
\end{align}
with $x(t) = x_1(t) -q_{33}\rho_{33}x_1(t-\tau_3)$, and $\Tilde{M}_4$ a piecewise continuous function.
\end{proof}
\section{A Lemma about common zeros of holomorphic functions}
\label{proof:lem_spect}
\begin{lem} \label{lem:controllability_spect}
Let $f_0, f_1, f_2$ be non–identically zero holomorphic functions on $\C$.
Assume that for all $s\in \C$,
\begin{equation}
\label{rank_assum_lem} 
\operatorname{rank}[f_0(s), f_1(s), f_2(s)] = 1.
\end{equation}
Then, for all $\varepsilon>0$, there exist $T\in (0,\varepsilon)$  and
$\alpha\in \R$ (both chosen in the complement of a non-dense countable set) such that for all $s\in \C$,
\[
\operatorname{rank}\Big[
f_0(s)-\alpha \!\int_0^T e^{-\nu s}\,d\nu\, f_2(s),\,
f_1(s)
\Big]=1.
\]
\end{lem}

\begin{proof}
For $i\in\{0,1,2\}$, define
\(
Z_i:=\{s\in\C:\,f_i(s)=0\}.
\)
Since each $f_i$ is holomorphic and not identically zero, $Z_i$ is
discrete, hence countable. In particular, $Z_0\cap Z_1$ is countable.

For $T>0$, define
\(
K_T(s):=\int_0^T e^{-\nu s}\,d\nu.
\)
For $s\neq0$,
\(
K_T(s)=\frac{1-e^{-Ts}}{s},
\)
and by continuity $K_T(0)=T\neq0$.
Thus
\(
Z_T:=\{s\in\C:\,K_T(s)=0\}
=\left\{\frac{2\pi i k}{T}:\,k\in\Z\setminus\{0\}\right\}.
\)

We first choose $T$ such that \(
Z_0\cap Z_1\cap Z_T=\varnothing.
\)
Let $T_b>0$ such that there exists $s\in Z_0\cap Z_1\cap Z_T$ satisfies
\(
s=\frac{2\pi i k}{T_b}
\)
for some $k\in\Z\setminus\{0\}$.
Hence such a $T_b$ must belong to the set
\[
\mathcal T_{\mathrm{bad}}
:=
\Bigl\{\,T_b>0:\,
\exists s\in Z_0\cap Z_1,\;
\exists k\in\Z\setminus\{0\}
\text{ such that }
T_b=\tfrac{2\pi i k}{s}
\Bigr\}.
\]
Since $Z_0\cap Z_1$ is countable, the set
$\mathcal T_{\mathrm{bad}}$ is countable.
Therefore we can choose
\(
T\in (0,\varepsilon)\setminus \mathcal T_{\mathrm{bad}.}
\)

Define
\[
\tilde f_T(s):=K_T(s)f_2(s).
\]

We now choose $\alpha$.

For $s\in Z_1\setminus (Z_2\cup Z_T)$ we have $\tilde f_T(s)\neq0$, and the equation
\(
f_0(s)-\alpha \tilde f_T(s)=0
\)
forces the unique value
\(
\alpha=\frac{f_0(s)}{\tilde f_T(s)}.
\)
Define the set of forbidden real values
\[
\mathcal A_{bad}:=
\Bigl\{
\tfrac{f_0(s)}{\tilde f_T(s)}:\,
s\in Z_1\setminus (Z_2\cup Z_T)
\Bigr\}
\cap\R.
\]
Since $Z_1\setminus (Z_2\cup Z_T)$ is countable, $\mathcal A_{bad}$ is countable.
Choose
\(
\alpha\in\R\setminus \mathcal A_{bad}.
\)

We now check the rank condition.
Fix $s\in\C$.

\medskip
\noindent
\emph{Case 1: $s\notin Z_1$.}
Then $f_1(s)\neq0$, hence the rank is $1$.

\medskip
\noindent
\emph{Case 2: $s\in Z_1\cap Z_2$.}
By \eqref{rank_assum_lem}, $f_0(s)\neq0$, so
\(
f_0(s)-\alpha\tilde f_T(s)=f_0(s)\neq0.
\)

\medskip
\noindent
\emph{Case 3: $s\in Z_1\cap Z_T$.}
Then $\tilde f_T(s)=0$, so
\(
f_0(s)-\alpha\tilde f_T(s)=f_0(s).
\)
Since $Z_0\cap Z_1\cap Z_T=\varnothing$, we have $f_0(s)\neq0$.

\medskip
\noindent
\emph{Case 4: $s\in Z_1\setminus (Z_2\cup Z_T)$.}
Then $\tilde f_T(s)\neq0$, and since $\alpha\notin \mathcal A_{bad}$,
\(
f_0(s)-\alpha\tilde f_T(s)\neq0.
\)

\medskip
In all cases, the two components cannot vanish simultaneously,
so the rank is $1$ for all $s\in\C$.
\end{proof}

\bibliographystyle{abbrv}
\bibliography{references}
               
\end{document}